\documentclass[12pt]{amsart}
\usepackage{amssymb}

\def\mod#1{{\ifmmode\text{\rm\ (mod~$#1$)}
\else\discretionary{}{}{\hbox{ }}\rm(mod~$#1$)\fi}}

\usepackage{amsmath}
\usepackage{graphicx} % \scalebox
\usepackage{environ}
\NewEnviron{myequation}{%
\begin{equation*}
\scalebox{0.5}{$\BODY$}
\end{equation*}
}

\usepackage{float}
\usepackage{numprint}
\usepackage{seqsplit}
\usepackage{lmodern}
\restylefloat{table}

\newtheorem{theorem}{Theorem}[section]
\newtheorem{lemma}[theorem]{Lemma}
\newtheorem{corollary}[theorem]{Corollary}
\newtheorem{proposition}[theorem]{Proposition}
\newtheorem{conjecture}[theorem]{Conjecture}
\theoremstyle{remark}

\numberwithin{equation}{section}

\def\mod#1{{\ifmmode\text{\rm\ (mod~$#1$)}
\else\discretionary{}{}{\hbox{ }}\rm(mod~$#1$)\fi}}

\usepackage{hyperref}

\begin{document}
\title{Arithmetic progressions in squarefull numbers }

\author{Prajeet Bajpai}
\address{Department of Mathematics, University of British Columbia, Vancouver, BC Canada V6T 1Z2}
\email{prajeet@math.ubc.ca}

\author{Michael A. Bennett}
\address{Department of Mathematics, University of British Columbia, Vancouver, BC Canada V6T 1Z2}
\email{bennett@math.ubc.ca}
\thanks{The second author was supported in part by a grant from NSERC}

\author{Tsz Ho Chan}
\address{Department of Mathematics, Kennesaw State University, Marietta, GA U.S.A. 30060}
\email{tchan4@kennesaw.edu}
%\thanks{Supported in part by a grant from NSERC}

%\author{Andrew Scoones}
%\address{Department of Mathematics, University of York, Heslington, York YO10 5DD, United Kingdom}
%\email{andrew.scoones@york.ac.uk}

%\thanks{Supported in part by a grant from NSERC}
%\subjclass{Primary 11D25, 11J86}

\date{\today}
%-----------------------------------------------------------------------------
\begin{abstract}
%-----------------------------------------------------------------------------
We answer a number of questions of Erd\H{o}s on the existence of arithmetic progressions in $k$-full numbers (i.e. integers with the property that every prime divisor necessarily occurs to at least the $k$-th power). Further, we deduce a variety of arithmetic constraints upon such progressions, under the assumption of the $abc$-conjecture of Masser and Oesterl\'e.
%-----------------------------------------------------------------------------
\end{abstract}

\maketitle

%----------------------------------------
\section{Introduction}
%----------------------------------------

Let $k \geq 2$ be an integer. We call a positive integer $n$  {\it $k$-full} if each prime divisor  $p$ of $n$ has the property that $p^k \mid n$. In case $k=2$, such integers are also known as 
 {\it squarefull} or {\it powerful} (with the latter notation being due to Golomb \cite{Go}). These integers were first studied in 1934 by Erd\H{o}s and Szekeres \cite{ErSz}, who proved that if we denote by $P_k(x)$ the number of $k$-full positive integers $\leq x$, then 
 $$
 P_2(x) = \frac{\zeta(3/2)}{\zeta(3)} x^{1/2} + O(x^{1/3})
 $$
and, more generally,
 $$
 P_k(x) = c_k x^{1/k} + O(x^{1/(k+1)}),
 $$
 where
 $$
 c_k = \prod_{p \; \mbox{\tiny{prime}}} \left( 1+\sum_{m=k+1}^{2k-1} p^{-m/k} \right).
 $$
 These asymptotics have been subsequently sharpened by Bateman and Grosswald \cite{BaGr}, extended to short intervals by, for example, Trifonov \cite{Tri}, Liu \cite{Liu} and the third author \cite{Chan-2020}, and to arithmetic progressions by Liu and Zhang \cite{LZ}, the third author and Tsang \cite{Chan-Tsang}, and the third author \cite{Chan-2015}.
 
Despite being a rather thin set, the $k$-full numbers satisfy somewhat less rigid local constraints than, for example, the comparably sized set of $k$-th perfect powers. In particular, while the squares contain three-term arithmetic progressions, they contain no four-term arithmetic progressions, and  the $k$-th perfect powers, for $k \geq 3$, contain no three-term arithmetic progressions. The squarefull numbers (and, more generally, $k$-full numbers), on the other hand, contain arithmetic progressions of arbitrary length (this was first observed by Makowski \cite{Ma}; see also Theorem 3 of \cite{Chan-2022}). Indeed, if we define
\begin{equation} \label{ex}
d=N=\prod_{p \leq m} p^k,
\end{equation}
where the product is over primes, then each of $N, N+d, \ldots, N+(m-1)d$ is $k$-full.

In 1975, Erd\H{o}s \cite{Erd-Man-76} posed a number of questions about $k$-full numbers, from a variety of viewpoints, discussing their additive properties (notably treated by Heath-Brown \cite{HB}), gaps between them, and the existence of arithmetic progressions in $k$-full numbers.
In the paper at hand, we will focus on the last of these topics.
Regarding this, 
Erd\H{o}s asked ``Are there infinitely many quadruples of relatively prime powerful numbers which form an arithmetic progression?''. He then defined $A(k)$ to be the largest integer for which there are $A(k)$ relatively prime $k$-full numbers in arithmetic progression, and $A^\infty (k)$ to be the largest integer for which there are infinitely many sets of $A^\infty (k)$ 
relatively prime $k$-full numbers in arithmetic progression. He conjectured, ``on rather flimsy probabilistic grounds'', that $A^\infty (k) = 2$ for all $k \geq 4$, but that $A^\infty (3) = 3$ (actually, he stated that $A^\infty (k) = 0$ for all $k \geq 4$, but this is a typographical error).

We will answer Erd\H{o}s' question in the affirmative, by constructing  infinite families of $4$-term arithmetic progressions of relatively prime powerful numbers, and $3$-term progressions of relatively prime $3$-full numbers. We will then show that Erd\H{o}s' conjectures on $A^\infty (k)$ follow (in somewhat nontrivial fashion) from the $abc$-conjecture of Masser and Oesterl\'e. In fact, we will prove the following results. The first of these may be viewed as saying, conjecturally at least, that $m$-term arithmetic progressions of  $k$-full numbers $N, N+d, \ldots, N+(m-1)d$ resemble those constructed in (\ref{ex}), in the sense that $\gcd (d,N)$ is ``large'' and that $d$ and $N$ are of roughly the same size.

\begin{theorem} \label{thm-kAPbounds}
Under the assumption of the $abc$-conjecture, if $m \geq 3$ and $k \geq 2$ are integers, then if $N$ and $d$ are positive integers  such that
$$
N, N+d, \ldots, N+(m-1)d
$$
are each $k$-full numbers in arithmetic progression, then, if $\epsilon > 0$, we have
\begin{equation} \label{fundy}
\gcd (d, N) \gg   (\max \{ d, N \} )^{\frac{m(1-1/k)  - 2}{m(1-1/k^2)  - 2} - \epsilon},
\end{equation}
\begin{equation} \label{fundy2}
d \gg N^{\frac{m(1-1/k)  - 1}{m(1-1/k^2)  - 1} - \epsilon}
\end{equation}
and
\begin{equation} \label{fundy3}
N \gg d^{\frac{m(1-1/k)+1/k-2}{m(1-1/k^2)+1/k-2} - \epsilon}.
\end{equation}
 If, further, $m \geq 2k-1$, we have the stronger inequalities
\begin{equation} \label{fundy4}
\gcd (d, N) \gg (\max \{ d, N \})^{\frac{m(1-1/k)  - 2}{m(1-1/(2k-1))  - 2} - \epsilon},
\end{equation}
\begin{equation} \label{fundy5}
d \gg N^{\frac{m(1-1/k)  - 1}{m(1-1/(2k-1))  - 1} - \epsilon}
\end{equation}
and
\begin{equation} \label{fundy6}
N \gg d^{\frac{m(1-1/k)+1/k-2}{m(1-1/(2k-1))+1/k-2} - \epsilon}.
\end{equation}
In each case, the implied constants depends upon $\epsilon, k$ and $m$.
\end{theorem}

Note that the bound (\ref{fundy}) is nontrivial for all pairs $(m,k)$ except for
\begin{equation} \label{except}
(m,k) \in \{ (3,2), (3,3), (4,2) \}.
\end{equation}
For these cases, we have the following.
\begin{theorem} \label{thm2}
For each pair of integers $(m,k)$ in (\ref{except}), 
there exist infinitely many $m$-term arithmetic progressions of $k$-full numbers $N, N+d, \ldots, N+(m-1)d$ with the property that $\gcd(d,N)=1$.
\end{theorem}

Combining these theorems yields the following as an immediate corollary.
\begin{corollary}
The $abc$-conjecture implies that
$$
A^\infty (2) = 4, \; A^\infty (3) = 3 \; \mbox{ and } \; A^\infty (k) = 2, \; \mbox{ for all } k \geq 4.
$$
\end{corollary}

We should note that the cases $(m,k)=(3,2)$ and $(3,3)$ in Theorem \ref{thm2} are quite routine, with the latter being essentially similar to earlier work of Cohn \cite{Coh} and Nitaj \cite{Ni} on the equation $x+y=z$ with $x, y, z$ $3$-full and coprime. Our construction in case $(m,k)=(4,2)$ is somewhat more delicate, and examples appear rather harder to come by. Indeed, the smallest example we know of a four term arithmetic progression of relatively prime squarefull numbers $N, N+d, N+2d, N+3d$  is with 

\begin{myequation}
\begin{array}{l}
N=1754165466244069308202217775999956378086940173344319563172002642821353893414665425909802276610337858551130228714908264 \\
092183795339077334979031700161975796410401371645661431360025397251281677069110162012277002086390746724745549802652454316540 \\
271195895308716314978016592356420740449112045901950803191243774768961934231660971591185887190573577668799607394109591204214 \\
604135648682900998727228226593126369291913667220946462186004173194306398691888158466413460639774157882482959802659110939652 \\
136228618175027376824043417358675378159581584277370148393502133757758008198566804907384712188429439324818835809177011739706 \\
061754762515269091685163078618119624125217884298638064713655378814564438037963281084546897486611052819897337172620291148315 \\
280417277153060964497088448139561749108173430365387040149246930709528091212924112879928417251967934790756742633841051639584 \\
495623456552417965269743585498263726964921039114000025167618337353033324722304307938697243424019540486387146646368242489545 \\
610330490914468389343108667203304362239733168424653279570390475273414115369168262097104720496240973641353670263031347258604 \\
611645652961570313664689
\end{array}
\end{myequation} 
\noindent and
\begin{myequation}
\begin{array}{l}
d=299898780205178410895351168299251416514240403797188258857731555237883704843231133768804957739478188249793937539671317969 \\
1601418926330109568204740736277680144034988467399677276604096892033111741572495408516424831194336534740791609445632167915867 \\
983229641072496202684204455967079877585222902298106764205135133998090972953520383628510708831082703460991606290744685985479 \\
7501394526552285824537988146406729182561111252050397430951389113906511754358596098416711685797387964702456409557198144316354 \\
4576448375185114822454053014407405653855625778505950788783405133132351025230212273281181402784825243367245601342351163938844 \\
7511602603027731160669778975066241789237483758668990872682732905466648286670718351547126005760567836664876228186443297040149 \\
2891437575965927769057874089301626715048200317876370602379876364947749811158867618438548845231108039817421342270142569518358 \\
8676480531310322826968151693434755211455987253377474811167725972237918157785225550208983526158278047706217051627806646721104 \\
9645486094775768146941903865819544168802079854092052035588751598337667447201595945847274679394267855049509300923765353012150 \\
88500325400536.
\end{array}
\end{myequation} 

The outline of this paper is as follows. In Section \ref{Sec-Rad}, we deduce a pair of upper bounds upon the radicals of $k$-full numbers. 
Section \ref{Sec-Id} is devoted to polynomial identities which we will require for our application of the $abc$-conjecture.
In Section \ref{Sec-T1}, we prove Theorem \ref{thm-kAPbounds}, while, in Section \ref{Sec-T2}, we construct families of $m$-term arithmetic progressions of $k$-full number to prove Theorem \ref{thm2}. In Section \ref{Sec-Sharp}, we restrict our attention to the case $k=2$  and construct families of $m$-tuples of squarefull numbers in arithmetic progression with $d$ relatively small compared to $N$, sharpening earlier work of the third author \cite{Chan-2022}. Finally, in Section \ref{Sec7}, we discuss the problem of determining, given $m$, the smallest value of $d$ such that each of $N, N+d, \ldots, N+(m-1)d$ is squarefull, generalizing a conjecture of Erd\H{o}s, Mollin and Walsh \cite{MoWa}.

%--------------------------------------------------------------------
\section{Upper bounds for radicals} \label{Sec-Rad}
%--------------------------------------------------------------------

If $n$ is a positive integer, we define the  {\it radical}  of $n$, $\mbox{Rad}(n)$ to be the product of the distinct prime divisors of $n$ (setting $\mbox{Rad}(1)=1$). In this section, our goal is to deduce upper bounds upon radicals of products of $k$-full integers, after certain common factors have been removed. Our first result applies to an individual radical.

\begin{lemma} \label{lem-5AP}
Let $k \geq 2$,  $n$ be a $k$-powerful number and $t$ be any $k$-powerful divisor of $n$. Then
\begin{equation} \label{powpow}
\mbox{Rad} \Bigl( \frac{n}{t} \Bigr) \le \frac{n^{1/k}}{t^{1/k^2}}.
\end{equation}
\end{lemma}

\begin{proof}
To prove \eqref{powpow}, it suffices to show that
\begin{equation} \label{powpow1}
1 \le \frac{1}{k} \nu_p(n) - \frac{1}{k^2} \nu_p(t) \; \; \text{ for all } \; \; p \mid \frac{n}{t}
\end{equation}
and
\begin{equation} \label{powpow2}
0 \le \frac{1}{k} \nu_p(n) - \frac{1}{k^2} \nu_p(t) \; \; \text{ for all } \; \; p \nmid \frac{n}{t}.
\end{equation}
Note that if $p \mid  n$, then $\nu_p(n) \ge k$ as $n$ is $k$-powerful.

\bigskip

Case 1: $\nu_p(n) = k$. Then, for any $k$-powerful divisor $t$ of $n$, we must have $\nu_p(t) = 0$ or $\nu_p(t) = k$. In the former case, $\nu_p(t) = 0$ and \eqref{powpow1} is satisfied. In the latter case, $p \nmid \frac{n}{t}$. Hence $\nu_p(\frac{n}{t}) = 0$ and \eqref{powpow2} is satisfied.

\bigskip

Case 2: $e := \nu_p(n) \ge k+1$. Then, for any $k$-powerful divisor $t$ of $n$, we must have $\nu_p(t) = 0$, $k \le \nu_p(t) \le e-1$, or $\nu_p(t) = e$. If $\nu_p(t) = 0$, then $p^e \, \| \, \frac{n}{t}, n$ and we have \eqref{powpow1}. If $k \le \nu_p(t) \le e-1$, then $p \mid \frac{n}{t}$ and
\[
\frac{1}{k} \nu_p(n) - \frac{1}{k^2} \nu_p(t) \ge \frac{e}{k} - \frac{e-1}{k^2} = \frac{(k-1)e+1}{k^2} \ge 1,
\]
so that \eqref{powpow1} is satisfied. If $\nu_p(t) = e$, then $p \nmid \frac{n}{t}$ and we have \eqref{powpow2}.

\end{proof}

For our applications, we will actually need to bound the radical of a product of terms. We prove the following.

\begin{lemma} \label{roadie}
Let $k \geq 2$ and $m \geq 3$ be integers, and assume that $m \geq 2k-1$.
Suppose that $N$ and $d$ are positive integers  such that
$$
N, N+d, \ldots, N+(m-1)d
$$
are $m$ $k$-full numbers in arithmetic progression. Choose integers $a_{i,j}$ such that
\begin{equation} \label{stamp}
N+jd =  \prod_{i=k}^{2k-1} (a_{i,j})^i 
\end{equation}
and write $t=\gcd(N,d)$.
Then 
\begin{equation} \label{raddy3}
\mbox{Rad}  \left( \frac{N}{t} \frac{N+d}{t}  \cdots \frac{N+(m-1)d}{t}  \right) 
\leq C_m  \, \frac{\prod_{j=0}^{m-1}  \prod_{i=k}^{2k-1} a_{i,j}}{t^{m/(2k-1)}},
\end{equation}
where
$$
C_m=  \prod_{\stackrel{p \leq m}{p \mbox{\tiny{ prime}}}}  p.
$$
\end{lemma}

\begin{proof}
The fact that one may represent a $k$-full number as a product like (\ref{stamp}) is an immediate consequence of the fact that every integer $\ell \geq k$ may be written as a linear combination of $k, k+1, \ldots, 2k-1$, with nonnegative coefficients.
To begin, notice that $m \equiv 0 \mod{2k-1}$ implies that the quantity
$$
\frac{\prod_{j=0}^{m-1}  \prod_{i=k}^{2k-1} a_{i,j}}{t^{m/(2k-1)}}
$$
is a rational number whose $(2k-1)$-st power is an integer and hence itself an integer.  Suppose that $p$ is a prime with $p > m$ and 
$$
\nu_p  \left( \frac{N}{t} \frac{N+d}{t}  \cdots \frac{N+(m-1)d}{t}  \right)  \geq 1.
$$
Then
\begin{equation} \label{blob}
\nu_p \left( \prod_{j=0}^{m-1}  \prod_{i=k}^{2k-1} a_{i,j}^i \right) - m \nu_p(t) \geq 1.
\end{equation}
If $2k-1 \mid m$, 
$$
\nu_p \left( \frac{\prod_{j=0}^{m-1}  \prod_{i=k}^{2k-1} a_{i,j}}{t^{m/(2k-1)}} \right)
= \frac{1}{2k-1} \nu_p \left(  \frac{\prod_{j=0}^{m-1}  \prod_{i=k}^{2k-1} a_{i,j}^{2k-1}}{t^{m}} \ \right)
$$
and so
$$
\nu_p \left( \frac{\prod_{j=0}^{m-1}  \prod_{i=k}^{2k-1} a_{i,j}}{t^{m/(2k-1)}} \right)
\geq \frac{1}{2k-1} \nu_p \left(  \frac{\prod_{j=0}^{m-1}  \prod_{i=k}^{2k-1} a_{i,j}^{i}}{t^{m}} \ \right) \geq \frac{1}{2k-1} > 0.
$$
It follows that 
$$
\nu_p \left( \frac{\prod_{j=0}^{m-1}  \prod_{i=k}^{2k-1} a_{i,j}}{t^{m/(2k-1)}} \right) \geq 1,
$$
as desired. Next suppose that $2k-1 \nmid m$, so that $m \geq 2k-1$ implies that $m \geq 2k$.
Since
$$
M= \frac{\prod_{j=0}^{m-1}  \prod_{i=k}^{2k-1} a_{i,j}^{2k-1}}{t^{m}} 
$$
is a positive integer, it suffices to prove that, for each prime $p>m$ with (\ref{blob}), we have
$$
\nu_p \left(  M \right) \geq 2k-1.
$$
Suppose that
$$
\nu_p \left( M \right) \leq 2k-2.
$$
From (\ref{blob}), we have $\nu_p(M) \geq 1$.
Notice further that if $2k-1 \mid \nu_p(t)$, then $2k-1 \mid \nu_p(M)$, a contradiction. We may thus write
$$
\nu_p(t)= (2k-1) d_1 + d_0, \; \mbox{ with } \; 1 \leq d_0 \leq 2k-2.
$$
From
$$
\nu_p \left( M  \right) =
\nu_p \left( \prod_{j=0}^{m-1}  \prod_{i=k}^{2k-1} a_{i,j}^i \right) - m \nu_p(t) 
+ \sum_{i=k}^{2k-2} (2k-1-i) \nu_p \left( \prod_{j=0}^{m-1}  a_{i,j} \right),
$$
we thus have
$$
 \sum_{i=k}^{2k-2} \sum_{j=0}^{m-1} (2k-1-i) \nu_p \left( a_{i,j} \right) \leq 2k-3.
 $$
The number of $j$ for which 
$$
\max \{ \nu_p( a_{k,j}), \nu_p(a_{k+1,j}), \ldots,  \nu_p(a_{2k-2,j}) \} \geq 1
$$
is thus at most $2k-3$. Since $m \geq 2k$, it follows that there are at least $3$ values of $j$ for which
$$
\nu_p( a_{k,j}) = \nu_p(a_{k+1,j}) =  \cdots = \nu_p(a_{2k-2,j}) =0,
$$
whereby
$$
\nu_p( N+jd) \equiv 0 \mod{2k-1}.
$$
It follows that
\begin{equation} \label{goop}
\nu_p( N+jd) > \nu_p(t),
\end{equation}
for each such $j$. Notice that if $\nu_p(d) > \nu_p(N)$, then we cannot have $\nu_p( N+jd) > \nu_p(t)$ for any $j$, while  if 
$\nu_p(d) < \nu_p(N)$, then $\nu_p( N+jd) > \nu_p(t)$ implies that $p \mid j$, whence, from $p > m$,  $j=0$, contradicting the existence of three values of $j$ satisfying (\ref{goop}).

We may thus suppose that 
$$
\nu_p(N)=\nu_p(d)=\nu_p(t),
$$
whence the existence of two values  $0 \leq j_1 < j_2  \leq m-1$ satisfying (\ref{goop}) implies that
$$
p^{\nu_p(d)+1}  \mid \left( (N+j_2d) - (N+j_1d) \right) = (j_2-j_1)d,
$$
and so $p \mid j_2-j_1$, whence $p \leq j_2-j_1 < m$, a contradiction.
 
\end{proof}

%----------------------------------------------------------------------------------------------------------
\section{Identities for consecutive terms in arithmetic progression} \label{Sec-Id}
%----------------------------------------------------------------------------------------------------------

In this section, we will derive a number of polynomial identities that enable us to package together $m$ $k$-full numbers in arithmetic progression, in such a way as to permit application of the $abc$-conjecture of Masser and Oesterl\'e. What follows is essentially Lemma 8.1 of Shorey and Tijdeman \cite{ST}; we provide its proof for completeness.

For any positive integers $n$ and $\ell$, let ${\scriptscriptstyle\begin{Bmatrix} n \\ \ell \end{Bmatrix}}$ be the number of partitions of $n$ labelled objects into $\ell$ non-empty unlabelled subsets. This is also known as Stirling's number of the second kind. Then, the number of surjections from an $n$-element set to an $\ell$-element set is given by $\ell ! \begin{Bmatrix} n \\ \ell \end{Bmatrix}$. Since the number of functions from an $n$-element set to an $j$-element set is $j^n$, we have
\[
\ell! \begin{Bmatrix} n \\ l \end{Bmatrix} = \sum_{j = 0}^{\ell} (-1)^{\ell - j} \binom{\ell}{j} j^n,
\]
by the inclusion-exclusion principle. In particular, when $n < \ell$,
\begin{equation} \label{zerosum}
\sum_{j = 0}^{\ell} (-1)^{\ell - j} \binom{\ell}{j} j^n = 0 = \sum_{j = 1}^{\ell} (-1)^{j - 1} \binom{\ell}{j} j^n,
\end{equation}
since there is no surjection from an $n$-element set to an $\ell$-element set. Clearly, $\begin{Bmatrix} \ell \\ \ell \end{Bmatrix} = 1$, whence
\begin{equation} \label{lsum}
\sum_{j = 0}^{\ell} (-1)^{\ell - j} \binom{\ell}{j} j^\ell = \ell! \; \; \text{ or } \; \; \sum_{j = 1}^{\ell} (-1)^{j - 1} \binom{\ell}{j} j^\ell = (-1)^{\ell-1} \ell! .
\end{equation}

\begin{lemma} \label{lem-identity}
Let  $\ell \ge 2$ and $d \ge 1$ be integers and set
$$
F_d(X) =  \prod_{\stackrel{1 \le j \le \ell}{j \text{ odd}}}  \bigl( X + j d \bigr)^{\binom{\ell}{j}} - \prod_{\stackrel{0 \le j \le \ell}{j \text{ even}}}  \bigl( X + j d \bigr)^{\binom{\ell}{j}}.
$$
Then $F_d(X) = d^\ell G_d(X)$, where $G_d(X)$ is a binary form in $X$ and $d$ with integer coefficients, of degree $2^{\ell-1}-\ell$, satisfying $G_0(1)=(\ell-1)!$.

\end{lemma}

\begin{proof}
 First, treating $d$ as a variable and expanding $F_d(X)$, one can clearly see that it is a homogeneous polynomial of the form
\begin{equation} \label{homog}
\sum_{i + j = 2^{\ell-1}} c_{i, j} d^{i} X^{j}
\end{equation}
as 
$$
\sum_{1 \le j \le \ell, \; j \text{ odd}} \binom{\ell}{j} = \sum_{0 \le j \le \ell, \; j \text{ even}} \binom{\ell}{j} = 2^{\ell-1}. 
$$
Consider the ratio
\[
R(x) = \frac{\prod_{1 \le j \le \ell, \; j \text{ odd}} \bigl( 1 + j x \bigr)^{\binom{\ell}{j}}}{\prod_{0 \le j \le \ell, \; j \text{ even}} \bigl( 1 + j x \bigr)^{\binom{\ell}{j}}}.
\]
Taking logarithms and using the Taylor series for $\log(1+x)$, we can expand $\log R(x)$ as 
\[
\sum_{j = 1}^{\ell} (-1)^{j-1} \binom{\ell}{j} \sum_{i = 1}^{\infty} (-1)^{i - 1} \frac{(j x)^i}{i} = \sum_{i = 1}^{\infty} (-1)^{i-1} \frac{x^i}{i} \sum_{j = 1}^{\ell} (-1)^{j - 1} \binom{\ell}{j} j^i.
\]
By \eqref{zerosum}, the coefficients of the terms  $x, x^2,  \ldots, x^{\ell-1}$ vanish whence
\[
\log R(x) = (\ell-1)! x^\ell + O_\ell (x^{\ell+1})
\]
by \eqref{lsum}. Exponentiating both sides and using the Taylor series for $e^x$, we find that
\[
R(x) = 1 + (\ell-1)! x^\ell + O_\ell (x^{\ell+1}).
\]
Substituting $x = d / N$ for some large parameter $N > d$, we have
\[
\frac{\prod_{1 \le j \le \ell, \; j \text{ odd}} \bigl( N + j d \bigr)^{\binom{\ell}{j}}}{\prod_{0 \le j \le \ell, \; j \text{ even}} \bigl( N + j d \bigr)^{\binom{\ell}{j}}} = 1 + (\ell-1)! \frac{d^\ell}{N^\ell} + O_\ell \Big( \frac{d^{\ell+1}}{N^{\ell+1}} \Bigr),
\]
i.e.
\[
F_d(N) = (\ell - 1)! N^{2^{\ell-1} - \ell} d^\ell + O_\ell \Bigl( N^{2^{\ell-1} - \ell - 1} d^{\ell+1} \Bigr),
\]
whence
\[
\lim_{N \rightarrow \infty} \frac{F_d(N)}{N^{2^{\ell-1} - \ell}} = (\ell-1)! d^\ell.
\]
Since $F_d(X)$ is a polynomial, it must have degree exactly $2^{\ell-1} - \ell$ in $X$ and leading coefficient $(\ell-1)! d^\ell$. Combining this with \eqref{homog}, we have the lemma.
\end{proof}

%---------------------------------------------------------------------------------------
\section{Proof of Theorem \ref{thm-kAPbounds}} \label{Sec-T1}
%---------------------------------------------------------------------------------------

We have our machinery in place to  prove Theorem \ref{thm-kAPbounds}, which depends fundamentally upon the $abc$-conjecture :
\begin{conjecture}[Masser and Oesterl\'e] \label{abc}
If $\epsilon > 0$, then there exists a constant $\kappa(\epsilon)>0$ such that if $a, b$ and $c$ are positive, coprime integers satisfying $a+b=c$, then
$$
c < \kappa(\epsilon) \mbox{Rad}(abc)^{1+\epsilon}.
$$
\end{conjecture}

Let  us
suppose that $m \geq 3$ and $k \geq 2$ are integers and that each of $N, N + d, \ldots, N + (m-1) d$ is a $k$-full number. By Lemma \ref{lem-identity}, with $\ell=m-1$, we have
\begin{equation} \label{kabc0}
\mathop{\prod_{1 \le j \le m-1}}_{j \text{ odd}} \bigl( N + j d \bigr)^{\binom{m-1}{j}} = \mathop{\prod_{0 \le j \le m-1}}_{j \text{ even}} \bigl( N + j d \bigr)^{\binom{m-1}{j}} + d^{m-1} G_d(N).
\end{equation}
Let
\[
t := \text{gcd}(N, d) = \text{gcd}(N, N + d),
\]
which is $k$-full  as both $N$ and $N + d$ are $k$-full. Then, $N_0 = N/t$ and $d_0 = d/t$ are relatively prime and \eqref{kabc0} becomes
\begin{equation} \label{kabc1}
\mathop{\prod_{1 \le j \le m-1}}_{j \text{ odd}} \bigl( N_0 + j d_0 \bigr)^{\binom{m-1}{j}} = \mathop{\prod_{0 \le j \le m-1}}_{j \text{ even}} \bigl( N_0 + j d_0 \bigr)^{\binom{m-1}{j}} +d_0^{m-1} G_{d_0}(N_0).
\end{equation}
Let
\[
D := \text{gcd} \biggl( \mathop{\prod_{1 \le j \le m-1}}_{j \text{ odd}} \bigl( N_0 + j d_0 \bigr)^{\binom{m-1}{j}}, \mathop{\prod_{0 \le j \le m-1}}_{j \text{ even}} \bigl( N_0 + j d_0 \bigr)^{\binom{m-1}{j}} \biggr)
\]
Since $\text{gcd}(d_0, N_0) = 1$, we have $\text{gcd}(d_0, N_0 + i d_0) = 1$ and so
\[
\text{gcd}(N_0 + i d_0, N_0 + j d_0) = \text{gcd}(N_0 + i d_0, (j - i) d_0) = \text{gcd}(N_0 + i d_0, j - i ),
\]
whence
$$
\text{gcd}(N_0 + i d_0, N_0 + j d_0) \leq  |j - i|
$$
for any $1 \le i, j \le m-1$. It follows that 
$$
D \le (m-1)^{(m-1)^2} \ll_m 1. 
$$
Since $t \mid N$ and $t \mid d$, we have $t \mid N, N+d, \ldots, N+ (m-1)d$.

Assume first that $d \leq N$. Then 
\begin{equation} \label{kpow}
\mbox{Rad}(N_0), \; \mbox{Rad}(N_0 + d_0),  \; \cdots,   \; \mbox{Rad}(N_0 + (m- 1) d_0) \ll \frac{N^{1/k}}{t^{1/k^2}},
\end{equation}
by Lemma \ref{lem-5AP}. Now, we apply Conjecture \ref{abc} to  \eqref{kabc1}, after dividing by $D$,
i.e. with
$$
a=\frac{1}{D}  \mathop{\prod_{0 \le j \le m-1}}_{j \text{ even}} \bigl( N_0 + j d_0 \bigr)^{\binom{m-1}{j}}, \; \;
b= \frac{1}{D} d_0^{m-1} G_{d_0}(N_0)
$$
and
$$
c= \frac{1}{D} \mathop{\prod_{1 \le j \le m-1}}_{j \text{ odd}} \bigl( N_0 + j d_0 \bigr)^{\binom{m-1}{j}} ,
$$
to conclude from \eqref{kpow} that
\begin{equation} \label{key}
\frac{N^{2^{m-2}}}{t^{2^{m-2}} D} \ll_{\epsilon, m} \biggl( \Bigl( \frac{N^{1/k}}{t^{1/k^2}} \Bigr)^{m} \cdot \frac{d}{t} \cdot \frac{N^{2^{m-2} - m+1}}{t^{2^{m-2} - m+1}} \biggr)^{1 + \epsilon},
\end{equation}
where $\epsilon > 0$ is arbitrary.
Since $d \leq N$, it follows that
$$
t \gg N^{\frac{m(1-1/k)  - 2}{m(1-1/k^2)  - 2} - \epsilon^\prime},
$$
where $\epsilon^\prime>0$ depends upon $\epsilon, k$ and $m$ (but approaches zero as $\epsilon$ does). 
Using just that $t \leq d$, (\ref{key}) implies that
$$
d \gg N^{\frac{m(1-1/k)  - 1}{m(1-1/k^2)  - 1} - \epsilon^{\prime\prime}},
$$
where $\epsilon^{\prime\prime}$ is suitably small and positive.

If, on the other hand, $d > N$, 
\begin{equation} \label{key2}
\frac{d^{2^{m-2}}}{t^{2^{m-2}} D} \ll_{\epsilon, m} \biggl( \Bigl( \frac{d^{1/k}}{t^{1/k^2}} \Bigr)^{m-1} \cdot \frac{N^{1/k}}{t^{1/k^2}}  \cdot \frac{d}{t} \cdot \frac{d^{2^{m-2} - m+1}}{t^{2^{m-2} - m+1}} \biggr)^{1 + \epsilon},
\end{equation}
and hence
$$
t \gg d^{\frac{m(1-1/k)  - 2}{m(1-1/k^2)  - 2} - \epsilon^\prime},
$$
and
$$
N \gg d^{\frac{m(1-1/k)  +1/k-2}{m(1-1/k^2)  +1/k-2} - \epsilon^{\prime\prime}}.
$$
Renaming $\epsilon$, we thus have inequalities (\ref{fundy}),  (\ref{fundy2})  and  (\ref{fundy3}).

Now suppose that $m \geq 2k-1$. If $d \leq N$, appealing to Lemma \ref{roadie}, we may thus replace inequality (\ref{key}) with
\begin{equation} \label{key3}
\frac{N^{2^{m-2}}}{t^{2^{m-2}} D} \ll_{\epsilon, m} \biggl( \Bigl( \frac{N^{1/k}}{t^{1/(2k-1)}} \Bigr)^{m} \cdot \frac{d}{t} \cdot \frac{N^{2^{m-2} - m+1}}{t^{2^{m-2} - m+1}} \biggr)^{1 + \epsilon},
\end{equation}
while $d > N$ implies that
\begin{equation} \label{key4}
\frac{d^{2^{m-2}}}{t^{2^{m-2}} D} \ll_{\epsilon, m} \biggl( \Bigl( \frac{d^{1/k}}{t^{1/(2k-1)}} \Bigr)^{m-1} \cdot \frac{N^{1/k}}{t^{1/(2k-1)}}  \cdot \frac{d}{t} \cdot \frac{d^{2^{m-2} - m+1}}{t^{2^{m-2} - m+1}} \biggr)^{1 + \epsilon}.
\end{equation}
Arguing as before leads to  (\ref{fundy4}), (\ref{fundy5}) and  (\ref{fundy6}).
This completes the proof of Theorem \ref{thm-kAPbounds}.

%--------------------------------------------------------------------------
\section{Proof of Theorem \ref{thm2}} \label{Sec-T2}
%--------------------------------------------------------------------------

\subsection{The case $(m,k)=(3,2)$}

There are a number of constructions to produce three term arithmetic progressions of coprime squarefull numbers. The simplest is to note that if $X, Y$ and $Z$ are coprime positive integers satisfying $X^2+Y^2=2Z^2$, with $X < Y$, then setting $N=X^2$ and $d = Z^2-X^2$ yields a three term arithmetic progression $N, N+d, N+2d$ of coprime squarefull numbers. A standard factoring argument in $\mathbb{Q}(\sqrt{-1})$ implies that all such solutions $X, Y, Z$ can be parametrized via
$$
X=a^2-b^2 \pm 2ab, \; Y= a^2-b^2 \mp 2ab \; \mbox{ and } \; Z = a^2+b^2,
$$
where $a$ and $b$ are integers with $\gcd (a,b)=1$,  and $a$ and $b$ have opposite parity.

\subsection{The case $(m,k)=(3,3)$}

Suppose that we have a solution to
\begin{equation} \label{triple}
X^3+Y^3=2 \cdot 3^4 Z^3,
\end{equation}
with $X$ and $Y$ coprime and distinct, and $Z \neq 0$. Note that this guarantees that $X \equiv -Y \mod{3}$. 
 Let $X_0=X$, $Y_0 = Y$, and, generally, for $i \geq 0$, set
$$
X_{i+1}  = X_i (X_i^3+2Y_i^3), \; \; Y_{i+1} = -Y_i (2X_i^3+Y_i^3)
$$
and
$$
 Z_{i+1} =  Z_i (X_i-Y_i)(X_i^2+X_iY_i+Y_i^2).
$$
It follows that $(X_{i+1}, Y_{i+1}, Z_{i+1})$ satifies equation (\ref{triple}). Also,
$$
|Z_{i+1}| = |X_i-Y_i| (X_i^2+X_iY_i+Y_i^2)  |Z_i| > |Z_i|.
$$
Suppose that $p$ is a prime with the property that $p \mid \gcd(X_1,Y_1)$. Since $X$ and $Y$ are coprime and odd, it follows that $p \mid X^3+2Y^3$ and $p \mid 2X^3+Y^3$, so that 
$p =3$, contradicting the fact that $X$ and $Y$ are coprime with  $X \equiv -Y \mod{3}$. We thus have
$\gcd (X_1, Y_1) =1$. Starting from the solution $X=37, Y=17, Z=7$ thus yields infinitely many distinct coprime solutions to equation  (\ref{triple}). 

To ensure that we have infinitely many such solutions with each of $X, Y$ and $Z$ positive (so that if we choose $N=X^3$ and $d=3^4 Z^3-X^3$, then $N, N+d$ and $N+2d$ are coprime, positive and cubefull), we must be slightly careful. We claim that at least one of the triples
$$
\pm (X_{i+h},Y_{i+h},Z_{i+h}), \; \mbox{ for } \; 0 \leq h \leq \log |X_i|
$$
consists entirely of positive integers. If this fails to be true for $(X_i,Y_i,Z_i)$ and $(-X_i,-Y_i,-Z_i)$, we may, without loss of generality (after possibly replacing $(X_i,Y_i,Z_i)$ with one of $(Y_i,X_i,Z_i)$, $(-X_i,-Y_i,-Z_i)$ or $(-Y_i,-X_i,-Z_i)$) suppose that $X_i > 0$, $Y_i < 0$ and $Z_i > 0$.
If, further, we have $X_i > \sqrt[3]{2} |Y_i| $, then $X_{i+1}$ and $Y_{i+1}$ are both positive. If  $X_i > 0$, $Y_i < 0$ and we have $X_i < \sqrt[3]{2} |Y_i| $, then, since $X_i^3 +Y_i^3 > 0$, it follows that $X_{i+1}<0$, and $Y_{i+1}>0$. Write
$$
u = \frac{Y_i}{X_i} = -1 +\delta, \; \mbox{ where } \; \frac{1}{|X_i|} \leq  \delta < 1-2^{-1/3}.
$$
Then
$$
\left| \frac{Y_{i+1}}{X_{i+1}} \right|  = |u| \cdot \left| \frac{u^3+2}{2u^3+1} \right|
$$
and, via the Mean Value Theorem, we can write
$$
\frac{u^3+2}{2u^3+1} = -1 - \delta_1, \; \mbox{ where } \; \delta_1 > 9 \delta.
$$
It follows that
$$
\left| \frac{Y_{i+1}}{X_{i+1}} \right|  = (1-\delta) (1+\delta_1) > (1-\delta)(1+9 \delta)  > 1+(9 \cdot 2^{-1/3}-1) \delta
$$
and so
$$
\left| \frac{Y_{i+1}}{X_{i+1}} \right|   > 1+6 \delta.
$$
Switching the roles of $X_{i+1}$ and $Y_{i+1}$ and iterating this argument, after $h$ iterations, either we have already found a pair of $X_{i+j}$ and  $Y_{i+j}$ with the same sign, $2 \leq j \leq h$, or  we have that $X_{i+h}> 0$, $Y_{i+h} < 0$ and 
$$
\left| \frac{X_{i+h}}{Y_{i+h}} \right|   > 1+6^h \delta.
$$
Since $\delta \geq  \frac{1}{|X_i|}$, if $h \geq \frac{\log |X_i|}{\log 6}$, it follows that
$$
\left| \frac{X_{i+h}}{Y_{i+h}} \right|   > 2 > \sqrt[3]{2},
$$
so that $X_{i+h+1}$ and $Y_{i+h+1}$ are both positive, as desired.

\subsection{The case $(m,k)=(4,2)$}

Let us write $N=N_0^2$ and $d=4d_0$, where we assume that $\gcd(N_0,6d_0)=1$. If we suppose that 
$$
N+d = N_0^2+4d_0 = r^2
$$
and
$$
N+2d = N_0^2+8d_0 = s^2,
$$
for integers $r$ and $s$,
then
$$
N_0^2+ s^2 = 2 r^2.
$$
We parametrize the solutions to this via
$$
N_0=a^2-b^2+2ab, \; s= a^2-b^2-2ab \; \mbox{ and } \; r = a^2+b^2,
$$
where we assume that $\gcd (a,b)=1$ and that $a$ and $b$ are of opposite parity. It follows that
$$
d_0 = ab (b^2-a^2)
$$
and so
$$
N+3d = N_0^2+12 d_0 = a^4 - 8a^3b + 2a^2 b^2 +8ab^3+b^4 = F(a,b).
$$

Our goal now is to show that there exist infinitely many coprime integers $a$ and $b$ of opposite parity for which $F(a,b)=73^3 c^2$, with $c$ an integer. To achieve this, we will first suppose that $F(a,b)=73z^2$ with $b \neq 0$. Writing $X=a/b$ and $Y=z/b^2$, it follows that
$$
X^4-8X^3+2X^2+8X+1=73 Y^2
$$
Noting that $F(2,-1)=73$, we write $X=-2 + \frac{1}{u}$, so that
$$
\frac{u^4 (X^4-8X^3+2X^2+8X+1)}{73} = u^4 - \frac{128}{73} u^3 + \frac{74}{73} u^2 - \frac{16}{73} u + \frac{1}{73}.
$$
The right-hand-side here is  
$$
\left(u^2 - \frac{64}{73} u + \frac{653}{73^2} \right)^2 + \left( - \frac{1680}{73^3} u  - \frac{37392}{73^4}  \right),
$$
whence taking $u=S/T$, $x=2 \cdot 73^2 T$ and $y=4 \cdot 73^3S$, it follows that
\begin{equation} \label{main-ell}
E \; \; : \; \; y^2 - 128xy - 3360y = x^3 - 2612 x^2 + 149568 x.
\end{equation}

Tracing this back to our original variables $a$ and $b$, we find that 
$$
\frac{y}{x} = 146u = \frac{146}{X+2}  = \frac{146}{a/b+2} = \frac{146b}{a+2b}.
$$
Our elliptic curve $E$  has $a$-invariants
$$
a_1= -128, \; a_2 = - 2612, \;  a_3= - 3360, \; a_4= 149568 \mbox{ and } a_6=0,
 $$
 and
$$
E(\mathbb{Q}) \cong \mathbb{Z}_2 \times \mathbb{Z}_2 \times \mathbb{Z}^2,
$$
with generators
$$
T_1= (-1176,  -73 \cdot 1008), \; T_2= (-300, -73 \cdot 240), 
$$
$$
P_1=(-976,  -49344) 
$$
and
$$
 P_2 = (-408, -30192).
$$

The smallest example we know with $F(a,b)=73^3 c^2$ and $c$ an integer is with $N$ and $d$ as given in the introduction, corresponding to the point
$$
P=14P_1-8P_2+T_1,
$$
where we have
$$
\begin{array}{c}
a=144921248310429651263484981703141139178769570556975502152668794143 \\
48005218241845175889247190278032641984251634277516443172330761983511 \\
75740111653955442577917805739155655813704091550263684508211384663921 \\
34833495031762207983301017434999503435028566544782377465850347794312 \\
258030799611
\end{array}
$$
and
$$
\begin{array}{c}
b=155947482664129442288576782709265033399180164546518342682672792886 \\
84237483873605039526561817244982673971927674062409153422746705211107 \\
18555020289971310406261959541458048923588962496661949596895481299529 \\
92257451377480920702977544294258857755793515667657642515368155441381 \\
225265260358.
\end{array}
$$

Write, for a positive integer $n$, 
$$
nP_1 = \left( \frac{\phi_n}{\psi_n^2}, \frac{\Omega_n}{\psi^3_n} \right).
$$
Then, on our original quartic,
we have
$$
\frac{2b}{a+2b} = \frac{\Omega_n}{73 \psi_n \phi_n}.
$$
Taking $x=-976$, we may define
$$
\psi_0 =0, \; \psi_1=1 \; \mbox{ and } \psi_2 =   2^5 \cdot 5 \cdot 11 \cdot 13,
$$
$$
\psi_3= 3 x^4+b_2 x^3+3b_4 x^2+3b_6 x+b_8,
$$
where
$$
b_2=a_1^2+4a_2= 5936, \; b_4=2a_4+a_1a_3 = 729216,
$$
$$
b_6=a_3^2= 11289600  \mbox{ and }  b_8=-a_1a_3a_4+a_2a_3^2-a_4^2 = - 116185227264.
$$
Thus
$$
\psi_3 =  -861920436224 = - 2^{13} \cdot 105214897.
$$
Finally,
$$
\frac{\psi_4}{\psi_2} =  2 x^6+b_2 x^5+5b_4 x^4+10b_6 x^3 + 10 b_8 x^2 +(b_2b_8-b_4b_6) x+ (b_4b_8-b_6^2).
$$
We generate $\psi_n$ for $n \geq 5$ via the recursion
\begin{equation} \label{recur}
\psi_{m+n} \psi_{m-n} = \psi_{m+1} \psi_{m-1} \psi_n^2-  \psi_{n+1} \psi_{n-1} \psi_m^2.
\end{equation}

We also have
$$
\phi_n = -976 \psi_n^2 - \psi_{n-1} \psi_{n+1}
$$
and
$$
\Omega_n = \frac{1}{2 \psi_n} \left( \psi_{2n} + \psi_n^2 \left( 128 \phi_n + 3360 \psi_n^2 \right)  \right).
$$

We compute that $\psi_n$ is periodic modulo $73$ with period $2628 = 36 \cdot 73$. 
 From the fact that $\phi_n$ and $\Omega_n$ satisfy
 $$
\phi_n =  -976 \psi_n^2-\psi_{n-1} \psi_{n+1} 
$$
and
$$
\Omega_n = \frac{1}{2 \psi_2} \left( \psi_{n+2} \psi_{n-1}^2 - \psi_{n-2} \psi_{n+1}^2 \right)
+\psi_n  \left( 64 \phi_n + 1680 \psi_n^2 \right),
$$
valid for $n \geq 1$ and $n \geq 2$, respectively,
 we may check that $\phi_n$ is periodic modulo $73$ with period length $1314=18 \cdot 73$, while $\Omega_n$ is periodic  mod $73$ with period length $876=12 \cdot 73$.
A short computation then reveals that
$$
\psi_n \phi_n \equiv 2 \Omega_n \mod{73} \; \mbox{ and } \; 73 \nmid \Omega_n,
$$
precisely when $n \equiv 39 \mod{73}$. For such $n$, we have
$$
 \frac{\psi_n \phi_n}{\Omega_n} \equiv 2 \mod{73},
 $$
whence
$$
\frac{a}{b} = \frac{146 \psi_n \phi_n}{\Omega_n}-2 \equiv 290 \mod{73^2}.
$$
The factorization of $X^4 -8 X^3+2 X^2 +8X+1$ as
$$
 (X-290)(X-2738)(X-2896)(X-4742) \mod{73^2}
$$
thus implies that 
$$
F(a,b) = a^4-8a^3b+2a^2b^2+8ab^3+b^4 \equiv 0 \mod{73^2},
$$
 and hence, since $F(a,b)=73z^2$, $73 \mid z$, i.e. $F(a,b)=73^3 c^2$ for some integer $c$.
 
It remains then to show that there are infinitely many indices $n \equiv 39 \mod{73}$ for which $a$ and $b$ are of opposite parity. Since
$$
\frac{a}{b} = \frac{146 \psi_n \phi_n}{\Omega_n}-2,
$$
in order to have $a$ even and $b$ odd, it suffices to show that
\begin{equation} \label{dancer}
\nu_2 ( \Omega_n ) \leq \nu_2 ( \psi_n )  + \nu_2 ( \phi_n ).
\end{equation}

We begin by proving the following result.
\begin{proposition} \label{goodies}
If $k$ is a positive integer, we have
$$
\nu_2(\psi_{4k+1}) = 13 k(2k+1), \; \; \nu_2(\psi_{4k-1}) = 13 k(2k-1), 
$$
$$
\nu_2(\psi_{4k+2}) = 26 k(k+1)+5
$$
and
$$
\nu_2(\psi_{4k})  = 
\left\{
\begin{array}{cl}
26k^2+4 &  \mbox{ if $k$ is odd } \\
26k^2+5 &  \mbox{ if $k \equiv 2 \mod{4}$} \\
\geq 26k^2+6  &  \mbox{ if $4 \mid k.$} \\
\end{array}
\right.
$$
\end{proposition}

\begin{proof}
We use induction
and identity (\ref{recur}).
We have
$$
\nu_2 (\psi_2) = 5, \; \nu_2 (\psi_3) = 13, \; \nu_2 (\psi_4) = 30, \;  \nu_2 (\psi_5) = 39,  \;  \nu_2 (\psi_6) = 57,
$$
$$
\nu_2 (\psi_7) = 78, \; \nu_2 (\psi_8) = 109, \; \nu_2 (\psi_9) = 130, \;  \nu_2 (\psi_{10}) = 161,  \;  \nu_2 (\psi_{11}) = 195,
$$
$$
\nu_2 (\psi_{12}) = 238, \; \nu_2 (\psi_{13}) = 273, \; \nu_2 (\psi_{14}) = 317, \;  \nu_2 (\psi_{15}) = 364
$$
and
$$
\nu_2 (\psi_{16}) = 422,
$$
by direct computation. 

Suppose, then, that we have our desired conclusion for, given $k \geq 3$, all $n \leq 4k+1$.  Consider first 
$$
\psi_{4k+\kappa} \psi_{4k-\kappa} = \psi_{4k+1} \psi_{4k-1} \psi_\kappa^2-  \psi_{\kappa+1} \psi_{\kappa-1} \psi_{4k}^2,
$$
for $\kappa \in \{ 2, 3, 5 \}$. We have
$$
\nu_2 \left( \psi_{4k+1} \psi_{4k-1} \psi_\kappa^2 \right) = 13 k(2k+1) + 13 k(2k- 1) + 2 \nu_2(\psi_\kappa) = 52k^2 + 2 \nu_2(\psi_\kappa),
$$
while
$$
\nu_2 \left(\psi_{\kappa+1} \psi_{\kappa-1} \psi_{4k}^2 \right) \geq 52k^2+8+ \nu_2(\psi_{\kappa+1} ) + \nu_2(\psi_{\kappa-1} ).
$$
For each $\kappa \in \{ 2, 3, 5 \}$, we find that 
$$
8+ \nu_2(\psi_{\kappa+1} ) + \nu_2(\psi_{\kappa-1} ) > 2 \nu_2(\psi_\kappa),
$$
whence
$$
\nu_2 \left( \psi_{4k+\kappa} \right) =52k^2 + 2 \nu_2(\psi_\kappa) - \nu_2 \left( \psi_{4k-\kappa} \right).
$$
Substituting  $\kappa=2$, $\kappa=3$ and $\kappa=5$, respectively, we thus have
$$
\nu_2 \left( \psi_{4k+2} \right) =52k^2 + 10 -  \left( 26k(k-1)+5 \right) = 26 k(k+1)+5,
$$
$$
\nu_2 \left( \psi_{4k+3} \right) =52k^2 + 26-  \left( 13(k-1)(2k-1)  \right) = 13 (k+1)(2k+1)
$$
and
$$
\nu_2 \left( \psi_{4k+5} \right) =52k^2 + 78 -  \left(  13 (k-1)(2k-3) \right) = 13 (k+1)(2k+3),
$$
in each case as desired.

It remains, then, to prove that $\nu_2(\psi_{4k+4}) =  26(k+1)^2+4$, if $k$ is even, while $\nu_2(\psi_{4k+4}) =   26(k+1)^2+5$, if $k \equiv 1 \mod{4}$, and $\nu_2(\psi_{4k+4}) \geq   26(k+1)^2+6$, if $k \equiv -1 \mod{4}$.
We appeal first to (\ref{recur}) 
with $m=4k+1, n=3$ to find that
$$
\psi_{4k+4} \psi_{4k-2} = \psi_{4k} \psi_{4k+2} \psi_3^2-  \psi_{4} \psi_{2} \psi_{4k+1}^2.
$$
If $k$ is even, 
$$
\nu_2 \left(  \psi_{4k} \psi_{4k+2} \psi_3^2 \right) \geq  26k^2+5 + 26 k(k+1)+5 + 26 = 52 k^2 +26k + 36,
$$
while
$$
\nu_2 \left(  \psi_{4} \psi_{2} \psi_{4k+1}^2 \right) =  52 k^2 +26k + 35,
$$
whence
$$
\nu_2 \left(  \psi_{4k+4} \right) = 26k^2+52k+30,
$$
as desired.
If, on the other hand, $k$ is odd, 
$$
\nu_2 \left(  \psi_{4k} \psi_{4k+2} \psi_3^2 \right) = 26k^2+4 + 26 k(k+1)+5 + 26 = 52 k^2 +26k + 35
$$
and
$$
\nu_2 \left(  \psi_{4} \psi_{2} \psi_{4k+1}^2 \right) = 30 + 5 + 26 k(2k+1) = 52k^2+26k+35,
$$
whence
$$
\nu_2 \left(  \psi_{4k+4} \right) \geq 52k^2+26k+36 - (26 k(k-1)+5) = 26k^2+52k+31.
$$
If $k \equiv 1 \mod{4}$, we can write $k=4j+1$ for integer $j$, so that $4k+4= 16j+8$, and 
$$
\psi_{16j+8} \psi_{8} = \psi_{8j+9} \psi_{8j+7} \psi_{8j}^2-  \psi_{8j+1} \psi_{8j-1} \psi_{8j+8}^2.
$$
We have
$$
\nu_2 \left(   \psi_{8j+9} \psi_{8j+7} \psi_{8j}^2 \right) =208(j+1)^2+ 2 \nu_2(\psi_{8j})
$$
and
$$
\nu_2 \left(   \psi_{8j+1} \psi_{8j-1} \psi_{8j+8}^2 \right) =208j^2+ 2 \nu_2(\psi_{8j+8}).
$$
Since $j$ and $j+1$ have opposite parity, it follows that 
$$
\nu_2 \left(  \psi_{4k+4} \right) = 26k^2+52k+31,
$$
as desired. Finally, if $k \equiv -1 \mod{4}$, say $k=4j-1$, we have $4k+4=16j$ and consider
$$
\psi_{16j} \psi_{8} = \psi_{8j+5} \psi_{8j+3} \psi_{8j-4}^2-  \psi_{8j-3} \psi_{8j-5} \psi_{8j+4}^2.
$$
We have
$$
\nu_2 \left(  \psi_{8j+5} \psi_{8j+3} \psi_{8j-4}^2 \right)= 52 (2j+1)^2 + 52 (2j-1)^2+8
$$
and
$$
\nu_2 \left(  \psi_{8j-3} \psi_{8j-5} \psi_{8j+4}^2 \right)= 52 (2j+1)^2 + 52 (2j-1)^2+8,
$$
whence 
$$
\nu_2 \left(  \psi_{4k+4} \right) \geq 26k^2+52k+32,
$$
again as claimed. 
\end{proof}

We will prove that inequality (\ref{dancer}) holds for all $n \equiv 4 \mod{16}$. Write  $n=16k+4$. We apply Proposition \ref{goodies} to find that
$$
\nu_2 ( \psi_{n-1} \psi_{n+1} )=52(4k+1)^2,
$$
while
$$
\nu_2 (976 \psi_n^2 ) > 52(4k+1)^2,
$$
so that, from $\phi_n =  -976 \psi_n^2-\psi_{n-1} \psi_{n+1}$,
$$
\nu_2 (\phi_n) = 52(4k+1)^2.
$$
We also find that
$$
\nu_2(\psi_n) = 26 (4k+1)^2+4 \; \mbox{ and } \; \nu_2(\psi_{2n}) = 26 (8k+2)^2+5,
$$
whence
$$
\nu_2 \left( 64 \psi_n \phi_n \right) = 78 (4k+1)^2+10,
$$
$$
\nu_2 \left( 1680 \psi_n^3 \right) = 78 (4k+1)^2+16
$$
and
$$
\nu_2 \left( \frac{\psi_{2n}}{2 \psi_n} \right) = 78 (4k+1)^2.
$$
It follows from
$$
\Omega_n = \frac{1}{2 \psi_n} \left( \psi_{2n} + \psi_n^2 \left( 128 \phi_n + 3360 \psi_n^2 \right)  \right)
$$
that
$$
\nu_2 \left( \Omega_n \right) = 78 (4k+1)^2,
$$
so that
$$
\nu_2 \left( \frac{\Omega_n}{ \psi_n \phi_n} \right) = -4,
$$
which implies (\ref{dancer}).

We have thus deduced the following.

\begin{proposition} \label{main-prop}
Define an elliptic curve $E$ via (\ref{main-ell}) and let
$$
P_1=(-976,  -49344) 
$$
be a point on $E$. Write, for a positive integer $n$, 
$$
nP_1 = \left( \frac{\phi_n}{\psi_n^2}, \frac{\Omega_n}{\psi^3_n} \right).
$$
Suppose that $a$ and $b$ are relatively prime integers satisfying
$$
\frac{a}{b} = \frac{146 \psi_n \phi_n}{\Omega_n}-2,
$$
where
$$
n \equiv 404 \mod{16 \cdot 73}.
$$
If we define
$$
N = \left( a^2-b^2+2ab \right)^2 \; \mbox{ and } \; d= 4 ab (b^2-a^2),
$$
then
$$
N, N+d, N+2d, N+3d
$$
are relatively prime squarefull numbers. More specifically, there exist infinitely many quadruples of relatively prime integers $x, y, z$ and $w$ such that
$$
N=x^2, \; N+d=y^2, \; N+2d=z^2 \; \mbox{ and } \; N+3d=73^3 w^2.
$$
\end{proposition}

%----------------------------------------------------------------------------------------------------
\section{Sharpness of Theorem  \ref{thm-kAPbounds} : $k=2$} \label{Sec-Sharp}
%----------------------------------------------------------------------------------------------------

For simplicity, for the remainder of this paper, we will focus our attention on the case $k=2$, i.e. on arithmetic progressions of squarefull aka powerful numbers. 
In \cite{Chan-2022}, the third author defines
$$
\theta_m = \liminf \left\{ \frac{\log d}{\log N}  :  N, d \in \mathbb{N}, \; N, N+d, \ldots, N+(m-1)d \mbox{ squarefull} \right\}
$$
and proves that, under the assumption of the $abc$-conjecture, one has
$$
\theta_3=\frac{1}{2}, \; \; \frac{1}{2}  \leq \theta_4 \leq \frac{9}{10}  \; \mbox{ and } \;  \frac{1}{2}  \leq \theta_m \leq 1- \frac{1}{10 \cdot 3^{m-5}},
$$
for $m \geq 5$.
In this section, we will  sharpen this result for all $m \geq 4$.
\begin{proposition} \label{thet}
Assuming the $abc$-conjecture, if $m \geq 4$ is an integer, then
$$
\frac{3m-6}{4m-6} \leq \theta_m \leq \frac{2m-4}{2m-3}.
$$
\end{proposition}

We note that the lower bound is an immediate consequence of Theorem  \ref{thm-kAPbounds}, specifically inequality (\ref{fundy5}), and depends upon the 
 $abc$-conjecture, while the upper bound is unconditional. To deduce these upper bounds, we will actually prove the following stronger result.

\begin{theorem} \label{4-term}
Let $m \geq 4$. Then there exist infinitely many positive integers $d$ and $N$ such that each of
$$
N, N+d, \ldots, N+(m-1)d
$$
is squarefull, and 
$$
d \ll N^{\frac{2m-4}{2m-3}} (\log N) ^{\frac{-2}{2m-3}}.
$$
\end{theorem}

\begin{proof}
Let us begin by defining integers $x_k$ and $y_k$ via
$$
x_k+ \sqrt{2} y_k = (1 +  \sqrt{2})^k,
$$
for nonnegative integer $k$. We will have need of the following result.
\begin{lemma} \label{Pelly}
If $j$ is a nonnegative integer and $k=3 \cdot 5^j \pm 1$,
we have 
\begin{equation} \label{pell}
y_k^2+1 \equiv 0 \mod{5^{j+1}}.
\end{equation}
\end{lemma}
\begin{proof} (of Lemma \ref{Pelly})
Begin by noting that if $j$ is a nonnegative integer, then
\begin{equation} \label{tupper}
y_{3 \cdot 5^j} \equiv 0 \mod{5^{j+1}}.
\end{equation}
To see this, apply induction in $j$, starting from the observation that $y_3=5$.  If we assume that (\ref{tupper}) holds, then, writing
$$
A= (1+ \sqrt{2})^{3 \cdot 5^j}  \; \mbox{ and } B = (1- \sqrt{2})^{3 \cdot 5^j} ,
$$
we have that 
$$
y_{3 \cdot 5^{j+1}} = \frac{A^5-B^5}{2 \sqrt{2}} = \frac{A-B}{2 \sqrt{2}} (A^4+A^3B+A^2B^2+AB^3+B^4)
$$
and hence
$$
y_{3 \cdot 5^{j+1}} = y_{3 \cdot 5^{j}} (A^4+A^3B+A^2B^2+AB^3+B^4).
$$
It remains then to show that
\begin{equation} \label{fiver}
A^4+A^3B+A^2B^2+AB^3+B^4 \equiv 0 \mod{5},
\end{equation}
a consequence of the fact that 
$$
A^4+A^3B+A^2B^2+AB^3+B^4 \equiv (A-B)^4 \mod{5}.
$$

To complete the proof of Lemma \ref{Pelly}, we note that, for any integer $k$, we have
$$
y_{2k-1} y_{2k+1} - y_{2k}^2 = 1.
$$
Appealing to this identity with $k=\frac{3 \cdot 5^j \pm 1}{2}$, (\ref{tupper}) thus implies (\ref{pell}), as desired.
\end{proof} 

Now let $m \geq 4$ be an integer and set
$$
x=y_{3 \cdot 5^j - 1},
$$
so that $2x^2+1$ is square, and we have
$$
x^2+1 \equiv 0 \mod{5^{j+1}}.
$$
For $3 \leq j \leq m$, define
$$
t_j = 2^{-\delta} \left( 2x^2+j \right),
$$
where $\delta = 0$ if $j$ is odd and $\delta=1$ if $j$ is even. Choose
$$
d= 2^2  \left( \prod_{j=3}^{m-1} t_j^2 \right)  \left( \frac{x^2+1}{5^{j+1}} \right)^2 
$$
and
$$
N= 2x^2 d,
$$
so that 
$$
N+jd= d(2x^2+j)
$$
is squarefull for each $0 \leq j \leq m-1$.
Since, for each $k$, we have
$$
y_k = \frac{(1+\sqrt{2})^k - (1-\sqrt{2})^k}{2 \sqrt{2}},
$$
it follows that
$$
x< \frac{(1+\sqrt{2})^{3 \cdot 5^j-1}}{2 \sqrt{2}}
$$
and hence 
$$
\log x < 2.65 \cdot 5^j < 5^{j+1}.
$$
Since $t_j \leq 2x^2+m$, for each $j$,  we 
thus have 
$$
d \ll \frac{x^{4(m-2)}}{\log^2 x}
$$
and so
$$
\frac{d}{N^{\frac{2m-4}{2m-3}}}  = d^{\frac{1}{2m-3}} \,  (2x^2)^{-\frac{2m-4}{2m-3}} \ll (\log x) ^{\frac{-2}{2m-3}} \ll (\log N) ^{\frac{-2}{2m-3}},
$$
i.e.
$$
d \ll N^{\frac{2m-4}{2m-3}} (\log N) ^{\frac{-2}{2m-3}}.
$$
This completes the proof of Theorem \ref{4-term}.
\end{proof}

\subsection{The case $m=3$} In  \cite{Chan-2022}, the third author asks whether ``it is possible to
construct infinitely many [three-term arithmetic progressions]  of powerful numbers with common difference $d=o(N^{1/2})$''? In this subsection, we will give a construction of a family of examples which, conjecturally at least, achieve this bound.
 Suppose that
$$
d=4 d_0 \; \mbox{ and } \; N=16N_0^2.
$$
In order  to have
$$
N+d=2^2 \cdot 3^3 \cdot t^2 \; \mbox{ and } N+2d = 2^3 \cdot s^2,
$$
with $s$ and $t$  integers,
we require that
$$
4 N_0^2 + d_0 = 3^3 t^2 \; \mbox{ and } \; 2N_0^2+d_0 = s^2,
$$
so that
\begin{equation} \label{crucial}
2N_0^2+s^2=3^3 t^2.
\end{equation}
We thus have
$$
s+\sqrt{-2} N_0 = (1+\sqrt{-2})^3 (u+\sqrt{-2}v)^2.
$$
Expanding,
$$
s+\sqrt{-2} N_0 = \left( -5 + \sqrt{-2} \right) \left(u^2 + 2 \sqrt{-2}u v -2v^2 \right),
$$
so that
$$
s= -5 u^2-4uv+10v^2
$$
and
$$
N_0=u^2-10uv -2v^2.
$$
It follows from $d_0 = s^2-2N_0^2$  that $d_0=F (u,v)$, where
$$
F(u,v)=23u^4 + 80u^3v - 276u^2v^2 - 160uv^3 + 92v^4,
$$
with corresponding roots $\theta_i$ to $F(x,1)=0$ given by
$$
\theta_1 = -5.4279\ldots, \; \theta_2= -0.8296\ldots, \; \theta_3=0.3684\ldots, \; \theta_4= 2.4107\ldots
$$
and discriminant $2^{20} \cdot 3^{18}$.
In order to make 
$$
d_0=23 (u-\theta_1v)(u-\theta_2v)(u-\theta_3v)(u-\theta_4v)
$$
as small as possible, relative to 
$$
|N_0| = \left| u^2-2v^2-10uv \right|,
$$
it thus suffices to choose  $u/v=p_k/q_k$  a convergent in the infinite simple continued fraction expansion to $\theta_i$ for some $i$.
By way of example,  if $u/v=p_k/q_k$ is a convergent to $\theta_3$, say, with corresponding partial quotient  $a_{k+1} \geq 60$ (so that $v \geq 19$), we have
$$
\left| \theta_3 - \frac{u}{v} \right| < \frac{1}{a_{k+1} v^2},
$$
whence
$$
|d_0| <  23.01 (\theta_3-\theta_1)(\theta_3-\theta_2)(\theta_4-\theta_3) \frac{v^2}{a_{k+1}},
$$
while
$$
|N_0| = \left| u^2-2v^2-10uv \right| > 0.99 \, (10 \theta_3-\theta_3^2+2 ) v^2.
$$
We thus have
$$
\frac{\sqrt{N}}{d} =\frac{|N_0|}{|d_0|} > \frac{0.99 (10 \theta_3-\theta_3^2+2)}{23.01 (\theta_3-\theta_1)(\theta_3-\theta_2)(\theta_4-\theta_3)} a_{k+1}.
$$
A short calculation reveals that, in order to have $d < \sqrt{N}$,  it thus suffices that $a_{k+1} \geq 60$. This happens  for
$$
k=5, 30, 122, 140, 206, 309, \ldots
$$

Sadly, while we would like to say that each $\theta_i$ has unbounded partial quotients (or, at the very least, infinitely many exceeding $60$), such a conclusion is currently not provable for even a single algebraic number of degree exceeding $2$. Conjecturally, however, the folklore belief is that such numbers behave in this respect like ``typical'' real numbers. In particular, their partial quotients are conjectured to satisfy a Gauss-Kuzmin distribution, and hence to be unbounded. To be more precise,
for almost all real numbers
 $\alpha$, by work of Philipp \cite{Phi},  if we define
$$
T_n(\alpha) = \max \{ a_\ell \; : \; 1 \leq \ell \leq n \},
$$
we have 
$$
\liminf_{n \rightarrow \infty} \frac{T_n(\alpha) \log \log n}{n} = \frac{1}{\log 2}.
$$
We thus expect that
$$
a_{k+1} \gg \frac{\log u}{\log \log \log u}
$$
infinitely often, which would imply that
$$ 
d = o \left(\sqrt{N} \right)
$$
infinitely often. More precisely, if any of the $\theta_i$ behaves like ``almost all'' real numbers regarding the partial quotients in their simple continued fraction expansions, then our construction provides infinitely many triples of squarefull numbers $N, N+d, N+2d$ with
$$ 
d \ll \frac{\sqrt{N} \log \log \log  N}{\log N}.
$$

\subsection{Computations}

If we simply enumerate all squarefull numbers up to $10^{12}$, say, then a brute-force search for examples of three term arithmetic progressions with $d < \sqrt{N}$ uncovers the following examples.

$$
\begin{array}{c|c|c}
d & N & \frac{\log d}{\log N} \\ \hline
2^2 \cdot 79 & 2^3 \cdot 3^6 \cdot 5^3 & 0.4263 \ldots \\
2^2 \cdot 11 \cdot 419 & 2^6 \cdot 3^2 \cdot 19^3 \cdot 47^2 & 0.4291 \ldots \\
2^2 \cdot 71 \cdot 647 & 2^2 \cdot 23^2 \cdot 11087^2 &  0.4611 \ldots \\
5^2 \cdot 13^2 \cdot 41 & 3^3  \cdot 5^2 \cdot 13^3 \cdot  317^2 & 0.4688 \ldots \\
2^2 \cdot 29^2 & 2^5 \cdot 5^2 \cdot 7^2 \cdot 29^2 & 0.4691 \ldots \\
2^2 \cdot 5^2 \cdot 41 \cdot 79 & 2^3 \cdot 5^2 \cdot 41761^2 & 0.4773 \ldots \\
2^2 \cdot 3^2 & 2^6 \cdot 3^3 & 0.4807 \ldots \\
3 \cdot 11^2 & 11^2 \cdot  37^2 & 0.4904 \ldots \\
2^2 \cdot 3 \cdot 1801 & 2^4 \cdot 6269^2 & 0.4926 \ldots \\
2^2 \cdot 1871 & 2^4 \cdot 2003^2 & 0.4962 \ldots \\
\end{array}
$$
Here, we are only listing ``primitive'' examples, i.e. those with $d$ minimal for fixed $d/N$. Analogous computation in case $m=4$ (where we observe that we do not know precisely the value $\theta_4$, only that $3/5 \leq \theta_4 \leq 4/5$) reveals the following examples with $ \frac{\log d}{\log N} < 0.7426$.

$$
\begin{array}{c|c|c}
d & N & \frac{\log d}{\log N} \\ \hline
139932 & 22358700 & 0.7001 \ldots \\
372100 & 90048200 & 0.7003 \ldots \\
10404 & 499392 & 0.7049 \ldots \\
744200 & 180096400 & 0.7112 \ldots \\
419796 & 67076100 & 0.7184 \ldots \\
6084 & 146016 & 0.7327 \ldots \\
127756 & 8821888 & 0.7352 \cdots \\
323276393476 & 3168108656064800 & 0.7425 \ldots \\
\end{array}
$$

As Theorem \ref{thm-kAPbounds} indicates, at least under the assumption of the $abc$-conjecture, if $m \geq 4$, then, additionally, $d$ cannot be too big relative to $N$. In case $m=4$,  the primitive pairs $d, N$ with $\frac{\log d}{\log N}$ largest that we know are as follows.

$$
\begin{array}{c|c|c}
d & N & \frac{\log d}{\log N} \\ \hline
129665228 & 21316 & 1.8741 \ldots \\
1083676 & 5324 & 1.6195 \ldots \\
9444665628 & 2008008 & 1.5826 \ldots \\
305569492668 & 25347564 & 1.5512 \ldots \\
810724 & 6728 & 1.5436 \ldots \\
8876268 & 38988 & 1.5134 \ldots \\
882683550 & 893025 & 1.5032 \ldots \\
%67240 & 1681 & 1.4966 \ldots \\
\end{array}
$$

There is nothing analogous to this situation (i.e. with $d$ bounded in terms of $N$) in case $m=3$. Indeed, if we take $N=1$, we can make $d+1$ and $2d+1$  both perfect squares by writing $d=x^2-1$ and imposing the condition that $2x^2-1=y^2$.

%-----------------------------------------------------------------------------------
\section{Concluding remarks : small values of $d$} \label{Sec7}
%----------------------------------------------------------------------------------

Let us define, for $m \geq 2$,
$$
d_m = \min \{ d   :  \mbox{ there exist } N \mbox{ with } N, N+d, \ldots, N+(m-1)d \mbox{ powerful}  \}.
$$
This is clearly nondecreasing in $m$ and $d_2=1$. We suspect, after computation, that $d_3=24$ (corresponding to $1, 25, 49$).
The Erd\H{o}s-Mollin-Walsh conjecture \cite{MoWa} (which appears to originate in a remark of Golomb \cite{Go})  is that $d_3 > 1$. 
In general, from (\ref{ex}), we have
$$
d_m \leq  \prod_{p \leq m} p^2,
$$
where the product is over primes and it seems likely that, for $m \geq 4$,
$$
d_m =  \prod_{p \leq m} p^2.
$$

To deduce a lower bound upon $d_m$ (in case $m \geq 4$), notice that if we suppose that $p \nmid d$, then 
$$
p \mid N(N+d) \cdots (N+(p-1) d).
$$
The fact that $m \geq p$ and the assumption that $N, N+d, \ldots, N+(m-1)d$ are all squarefull therefore implies that
$$
p^2 \mid N(N+d) \cdots (N+(p-1)d).
$$
But then 
$$
p \, \| \,  (N+pd) (N+(p+1) d) \cdots (N+(2p-1)d),
$$
a contradiction if $m \geq 2p$. It follows that 
$$
\prod_{p \leq m/2} p \mid d,
$$
so that
 $$
m \ll \log d_m \ll m.
$$

Returning to the  Erd\H{o}s-Mollin-Walsh conjecture, it is interesting to note that various authors \cite{AkMu}, \cite{Coh}, \cite{Mo}, \cite{Ni}, \cite{Rib}, \cite{SW} misattribute this conjecture (and related questions on squarefull numbers in arithmetic progression) to  papers of Erd\H{o}s \cite{Erd-Eur}, \cite{Erd-Deb-76}, while such problems are actually only discussed in \cite{Erd-Man-76}. This highlights both past difficulties of journal access, and also Erd\H{o}s' remarkable productivity!

%--------------------------------------------------------------------------------------------------------
\bibliographystyle{amsplain}

%%%%%%%%%%%%%%%%%%%%%%%%%%

  \end{document}